\documentclass[11pt]{article}
\usepackage{amsmath}
\usepackage{amssymb}
\usepackage{amsthm}
\usepackage[usenames]{color}
\usepackage{amscd}
\usepackage{dsfont}
\usepackage{indentfirst}

\usepackage[colorlinks=true,linkcolor=blue,filecolor=red,
citecolor=webgreen]{hyperref}
\definecolor{webgreen}{rgb}{0,.5,0}

\def\N{{\mathds{N}}}
\def\Z{{\mathds{Z}}}

\def\1{{\bf 1}}
\def\nr{{\trianglelefteq}}

\def\lcm{\operatorname{lcm}}

\newtheorem{theorem}{Theorem}[section]
\newtheorem{thm}[theorem]{Theorem}

\newtheorem{cor}[theorem]{Corollary}

\newtheorem{prop}[theorem]{Proposition}
\newtheorem{remark}[theorem]{Remark}

\begin{document}
\title{{\bf Subgroups of finite abelian groups having rank two via Goursat's lemma}}
\author{L\'aszl\'o T\'oth}
\date{}
\maketitle

\centerline{Tatra Mt. Math. Publ. {\bf 59} (2014), 93--103}
\begin{abstract} Using Goursat's lemma for groups, a
simple representation and the invariant factor decompositions of the
subgroups of the group $\Z_m \times \Z_n$ are deduced, where $m$ and
$n$ are arbitrary positive integers. As consequences, explicit
formulas for the total number of subgroups, the number of subgroups
with a given invariant factor decomposition, and the number of
subgroups of a given order are obtained.
\end{abstract}



\section{Introduction}

Let $\Z_m$ denote the additive group of residue classes modulo $m$
and consider the direct product $\Z_m \times \Z_n$, where $m,n \in
\N:=\{1,2,\ldots\}$ are arbitrary. Note that this group is
isomorphic to $\Z_{\gcd(m,n)}\times \Z_{\lcm(m,n)}$. If
$\gcd(m,n)=1$, then it is cyclic, isomorphic to $\Z_{mn}$. If
$\gcd(m,n)>1$, then $\Z_m \times \Z_n$ has rank two. We recall that
a finite abelian group of order $>1$ has rank $r$ if it is
isomorphic to $\Z_{n_1} \times \cdots \times \Z_{n_r}$, where
$n_1,\ldots,n_r \in \N\setminus \{1\}$ and $n_j \mid n_{j+1}$ ($1\le
j\le r-1$), which is the invariant factor decomposition of the given
group. Here the number $r$ is uniquely determined and represents the
minimal number of generators of the group. For general accounts on
finite abelian groups see, e.g., \cite{Mac2012, Rot1995}.

In this paper we apply Goursat's lemma for groups, see Section
\ref{section_Goursat_lemma}, to derive a simple representation and
the invariant factor decompositions of the subgroups of $\Z_m \times
\Z_n$ (Theorem \ref{prop_repres}). These are new results, as far as we know.
Then, we deduce as consequences, by purely number theoretical arguments, explicit formulas for the
total number of subgroups of $\Z_m \times \Z_n$ (Theorem \ref{prop_total_number_subgroups}), the number of its subgroups
of a given order (Theorem \ref{prop_total_number_subgroups_ord}) and the number of subgroups with a given invariant factor
decomposition (Theorem \ref{prop_number_subgroups_type}, which is another new result). The number of cyclic subgroups (of a given order)
is also treated (Theorems \ref{prop_number_cyclic_subgroups} and \ref{prop_number_subgroups_cyclic_ord}).
Furthermore, in Section \ref{section_table} a table for the subgroups of the group $\Z_{12} \times \Z_{18}$ is given to illustrate
the applicability of our identities.

The results of Theorems \ref{prop_total_number_subgroups} and
\ref{prop_total_number_subgroups_ord} generalize and put in more
compact forms those of G.~C\u{a}lug\u{a}reanu \cite{Cal2004},
J.~Petrillo \cite{Pet2011} and M.~T\u{a}rn\u{a}uceanu
\cite{Tar2010}, obtained for $p$-groups of rank two, and included in
Corollaries \ref{cor_number_subgroups} and
\ref{cor_number_subgroups_order}. We remark that both the papers
\cite{Cal2004} and \cite{Pet2011} applied Goursat's lemma for groups
(the first one in a slightly different form), while the paper
\cite{Tar2010} used a different approach based on properties of
certain attached matrices.

Another representation of the subgroups of $\Z_m \times \Z_n$, and
the formulas of Theorems \ref{prop_total_number_subgroups}, \ref{prop_total_number_subgroups_ord} and \ref{prop_number_cyclic_subgroups},
but not Theorem \ref{prop_number_subgroups_type}, were also derived  in \cite{HHTW2012} using different group theoretical arguments.
That representation and the formula of Theorem \ref{prop_total_number_subgroups} was generalized to the case of the subgroups of the group
$\Z_m \times \Z_n \times \Z_r$ ($m,n,r\in \N$) \cite{HamTot2013}, using similar arguments, which are different from those of the present paper.

Note that in the case $m=n$ the subgroups of $\Z_n \times \Z_n$ play an
important role in the field of applied time-frequency analysis (cf. \cite{HHTW2012}). See \cite{NowTot2013} for
asymptotic results on the number of subgroups of $\Z_m \times \Z_n$.

Throughout the paper we use the following standard notations:
$\tau(n)$ is the number of the positive
divisors of $n$, $\phi$ denotes Euler's totient function, $\mu$ is
the M\"obius function, $*$ is the Dirichlet convolution of
arithmetic functions.

\section{Goursat's lemma for groups} \label{section_Goursat_lemma}

Goursat's lemma for groups (\cite[p.\ 43--48]{Gou1889}) can be
stated as follows:

\begin{prop} \label{prop_Goursat}
Let $G$ and $H$ be arbitrary groups. Then there is a bijection
between the set $S$ of all subgroups of $G \times H$ and the set $T$
of all $5$-tuples $(A, B, C, D, \Psi)$, where $B\, \nr \, A \le G$,
$D\, \nr \, C\le H$ and $\Psi: A/B \to C/D$ is an isomorphism (here
$\le$ denotes subgroup and $\nr$ denotes normal subgroup). More
precisely, the subgroup corresponding to $(A, B, C, D, \Psi)$ is
\begin{equation} \label{repres_K}
K= \{ (g,h)\in A\times C: \Psi(gB)=hD\}.
\end{equation}
\end{prop}

\begin{cor} \label{cor_Goursat} Assume that $G$ and $H$ are finite groups and that the subgroup
$K$ of $G \times H$ corresponds to the $5$-tuple $(A_K, B_K, C_K, D_K, \Psi_K)$ under this bijection.
Then one has $|A_K|\cdot |D_K| =|K|=|B_K|\cdot |C_K|$.
\end{cor}

For the history, proof, discussion, applications and a
generalization of Goursat's lemma see
\cite{AndCam2009,BauSenZve2011,CraWal1975,Lam1958,Pet2009,Pet2011}.
Corollary \ref{cor_Goursat} is given in \cite[Cor.\ 3]{CraWal1975}.

\section{Representation of the subgroups of \texorpdfstring{$\Z_m \times \Z_n$}{ZmxZn}}

For every $m,n\in \N$ let
\begin{equation} \label{def_J_m_n}
J_{m,n}:=\left\{(a,b,c,d,\ell)\in \N^5: a\mid m, b\mid a, c\mid n, d\mid
c, \frac{a}{b}=\frac{c}{d}, \right.
\end{equation}
\begin{equation*} \left.
\ell \le \frac{a}{b}, \, \gcd\left(\ell,\frac{a}{b} \right)=1\right\}.
\end{equation*}

Using the condition $a/b=c/d$ we deduce
$\lcm(a,c)=\lcm(a,ad/b)=\lcm(ad/d,ad/b)=ad/\gcd(b,d)$. That is,
$\gcd(b,d)\cdot \lcm(a,c)=ad$. Also, $\gcd(b,d)\mid \lcm(a,c)$.

For $(a,b,c,d,\ell)\in J_{m,n}$ define
\begin{equation} \label{def_K}
K_{a,b,c,d,\ell}:= \left\{\left(i\frac{m}{a}, i\ell \frac{n}{c}+j\frac{n}{d}\right): 0\le i\le a-1, 0\le
j\le d-1\right\}.
\end{equation}

\begin{thm} \label{prop_repres} Let $m,n\in \N$.

i) The map $(a,b,c,d,\ell)\mapsto K_{a,b,c,d,\ell}$ is a bijection
between the set $J_{m,n}$ and the set of subgroups of $(\Z_m \times
\Z_n,+)$.

ii) The invariant factor decomposition of the subgroup
$K_{a,b,c,d,\ell}$ is
\begin{equation} \label{H_isom}
K_{a,b,c,d,\ell} \simeq \Z_{\gcd(b,d)} \times \Z_{\lcm(a,c)}.
\end{equation}

iii) The order of the subgroup $K_{a,b,c,d,\ell}$ is $ad$ and its
exponent is $\lcm(a,c)$.

iv) The subgroup $K_{a,b,c,d,\ell}$ is cyclic if and only if
$\gcd(b,d)=1$.
\end{thm}

The Figure represents the subgroup $K_{6,2,18,6,1}$ of
$\Z_{12}\times \Z_{18}$. It has order $36$ and is isomorphic to
$\Z_2\times \Z_{18}$.

\medskip
\centerline{$%
\begin{array}{ccccccccccccc}
17 & \cdot & \cdot & \cdot & \cdot & \bullet & \cdot & \cdot & \cdot & \cdot & \cdot & \bullet & \cdot \\
16 & \cdot & \cdot & \bullet & \cdot & \cdot & \cdot & \cdot & \cdot & \bullet & \cdot & \cdot & \cdot \\
15 & \bullet & \cdot & \cdot & \cdot & \cdot & \cdot & \bullet & \cdot & \cdot & \cdot & \cdot & \cdot \\
14 & \cdot & \cdot & \cdot & \cdot & \bullet & \cdot & \cdot & \cdot & \cdot & \cdot &
\bullet & \cdot \\
13 & \cdot & \cdot & \bullet & \cdot & \cdot & \cdot & \cdot & \cdot & \bullet & \cdot & \cdot & \cdot \\
12 & \bullet & \cdot & \cdot & \cdot & \cdot & \cdot & \bullet & \cdot & \cdot & \cdot & \cdot & \cdot \\
11 & \cdot & \cdot & \cdot & \cdot & \bullet & \cdot & \cdot & \cdot & \cdot & \cdot & \bullet & \cdot \\
10 & \cdot & \cdot & \bullet & \cdot & \cdot & \cdot & \cdot & \cdot & \bullet & \cdot & \cdot & \cdot \\
9 & \bullet & \cdot & \cdot & \cdot & \cdot & \cdot & \bullet & \cdot & \cdot & \cdot & \cdot & \cdot \\
8 & \cdot & \cdot & \cdot & \cdot & \bullet & \cdot & \cdot & \cdot & \cdot & \cdot & \bullet & \cdot \\
7 & \cdot & \cdot & \bullet & \cdot & \cdot & \cdot & \cdot & \cdot & \bullet & \cdot & \cdot & \cdot \\
6 & \bullet & \cdot & \cdot & \cdot & \cdot & \cdot & \bullet & \cdot & \cdot & \cdot & \cdot & \cdot \\
5 & \cdot & \cdot & \cdot & \cdot & \bullet & \cdot & \cdot & \cdot & \cdot & \cdot & \bullet & \cdot \\
4 & \cdot & \cdot & \bullet & \cdot & \cdot & \cdot & \cdot & \cdot & \bullet & \cdot & \cdot & \cdot \\
3 & \bullet & \cdot & \cdot & \cdot & \cdot & \cdot & \bullet & \cdot & \cdot & \cdot & \cdot & \cdot \\
2 & \cdot & \cdot & \cdot & \cdot & \bullet & \cdot & \cdot & \cdot & \cdot & \cdot &
\bullet & \cdot \\
1 & \cdot & \cdot & \bullet & \cdot & \cdot & \cdot & \cdot & \cdot & \bullet & \cdot & \cdot & \cdot \\
0 & \bullet & \cdot & \cdot & \cdot & \cdot & \cdot & \bullet & \cdot & \cdot & \cdot & \cdot & \cdot \\
& 0\text{ } & 1\text{ } & 2\text{ } & 3\text{ } & 4\text{ } & 5\text{ } & 6%
\text{ } & 7\text{ } & 8\text{ } & 9\text{ } & 10 & 11%
\end{array}%
$} \vskip2mm

\centerline{Figure}
\medskip

\begin{proof} i) Apply Goursat's lemma for the groups $G=\Z_m$ and
$H=\Z_n$. We only need the following simple additional properties:

$\bullet$ all subgroups and all quotient groups of $\Z_n$ ($n\in \N$) are
cyclic;

$\bullet$ for every $n\in \N$ and every $a\mid n$, $a\in \N$, there is
precisely one (cyclic) subgroup of order $a$ of $\Z_n$;

$\bullet$ the number of automorphisms of $\Z_n$ is $\phi(n)$ and they can be represented
as $f:\Z_n \to \Z_n$, $f(x)=\ell x$, where $1\le \ell \le n$,
$\gcd(\ell,n)=1$.

With the notations of Proposition \ref{prop_Goursat}, let $|A|=a$,
$|B|=b$, $|C|=c$, $|D|=d$, where $a\mid m$, $b\mid a$, $c\mid n$,
$d\mid c$. Writing explicitly the corresponding subgroups and
quotient groups we deduce
\begin{equation*}
A=\langle m/a \rangle = \left\{0,\frac{m}{a}, 2\frac{m}{a}, \ldots, (a-1)\frac{m}{a} \right\} \le \Z_m,
\end{equation*}
\begin{equation*}
B=\langle m/b \rangle = \left\{0,\frac{m}{b}, 2\frac{m}{b}, \ldots, (b-1)\frac{m}{b} \right\}\le A,
\end{equation*}
\begin{equation*}
A/B=\left\langle \frac{m}{a}+B \right\rangle = \left\{B,\frac{m}{a}+B, 2\frac{m}{a}+B, \ldots, \left(\frac{a}{b}-1\right)\frac{m}{a}+B \right\},
\end{equation*}
and similarly
\begin{equation*}
C=\langle n/c \rangle = \left\{0,\frac{n}{c}, 2\frac{n}{c}, \ldots, (c-1)\frac{n}{c} \right\}\le \Z_n,
\end{equation*}
\begin{equation*}
D=\langle n/d \rangle = \left\{0,\frac{n}{d}, 2\frac{n}{d}, \ldots,
(d-1)\frac{n}{d} \right\}\le C,
\end{equation*}
\begin{equation*}
C/D=\left\langle \frac{n}{c}+D \right\rangle = \left\{D,\frac{n}{c}+D, 2\frac{n}{c}+D, \ldots, \left(\frac{c}{d}-1\right)\frac{n}{c}+D \right\}.
\end{equation*}

Now, in the case $a/b=c/d$ the values of the isomorphisms $\Psi: A/B \to C/D$ are
\begin{equation*}
\Psi\left(i\frac{m}{a}+B\right)=i\ell \frac{n}{c} +D, \qquad 0\le i\le \frac{a}{b}-1,
\end{equation*}
where $1\le \ell \le a/b$, $\gcd(\ell ,a/b)=1$. Using
\eqref{repres_K} we deduce that the corresponding subgroup is
\begin{equation*}
K= \left\{ \left(i\frac{m}{a}, k\frac{n}{c}\right) \in A \times C: \Psi\left(i\frac{m}{a}+B\right)=k\frac{n}{c}+ D \right\}
\end{equation*}
\begin{equation*}
= \left\{ \left(i\frac{m}{a},k\frac{n}{c}\right): 0\le i\le a-1, 0\le k\le c-1, i\ell \frac{n}{c}+D =
k\frac{n}{c}+D \right \},
\end{equation*}
where the last condition is equivalent, in turn, to $kn/c \equiv i\ell n/c$
(mod $n/d$), $k\equiv i\ell$ (mod $c/d$), and finally $k=i\ell+jc/d$,
$0\le j\le d-1$. Hence,
\begin{equation*}
K= \left\{ \left(i\frac{m}{a},\left(i\ell +j\frac{c}{d}\right)\frac{n}{c} \right): 0\le i\le a-1, 0\le j\le d-1\right\},
\end{equation*}
and the proof of the representation formula is complete.

ii-iii) It is clear from \eqref{def_K} that $|K_{a,b,c,d,\ell}|=ad=bc$ (or cf.
Corollary \ref{cor_Goursat}). Next we deduce the exponent of
$K_{a,b,c,d,\ell}$. According to \eqref{def_K} the subgroup
$K_{a,b,c,d,\ell}$ is generated by the elements $(0,n/d)$ and
$(m/a,\ell n/c)$. Here the order of $(0,n/d)$ is $d$. To obtain the
order of $(m/a,\ell n/c)$ note the following properties:

(1) $m\mid r(m/a)$ if and only if $m/\gcd(m,m/a) \mid r$ if and only
if $a\mid r$, and the least such $r\in \N$ is $a$,

(2) $n\mid t(\ell n/c)$ if and only if $n/\gcd(n,\ell n/c) \mid t$
if and only if $c/\gcd(\ell,c) \mid t$, and the least such $t\in \N$
is $c/\gcd(\ell,c)$.

Therefore the order of $(m/a,\ell n/c)$ is $\lcm(a,c/\gcd(\ell,c))$.
We deduce that the exponent of $K_{a,b,c,d,\ell}$ is
\begin{equation*}
\lcm \left(d,\lcm \left(a,\frac{c}{\gcd(\ell,c)}\right)\right) =
\lcm \left(d,a,\frac{c}{\gcd(\ell,c)}\right)
\end{equation*}
\begin{equation*}
= \lcm \left(\frac{ac}{ac/d}, \frac{ac}{c},
\frac{ac}{a\gcd(\ell,c)}\right) =
\frac{ac}{\gcd(ac/d,c,a\gcd(\ell,c))}
\end{equation*}
\begin{equation*}
= \frac{ac}{\gcd(c,a\gcd(c/d,\gcd(\ell,c)))}
= \frac{ac}{\gcd(c,a\gcd(\ell,\gcd(c/d,c)))}
\end{equation*}
\begin{equation*}
= \frac{ac}{\gcd(c,a\gcd(\ell ,c/d))}
=\frac{ac}{\gcd(a,c)}=\lcm(a,c),
\end{equation*}
using that $\gcd(\ell,c/d)=1$, cf. \eqref{def_J_m_n}.

Now $K_{a,b,c,d,\ell}$ is a subgroup of the abelian group $\Z_m
\times \Z_n$ having rank $\le 2$. Therefore, $K_{a,b,c,d,\ell}$ has
also rank $\le 2$. That is, $K_{a,b,c,d,\ell} \simeq \Z_u \times
\Z_v$ for unique $u$ and $v$, where $u\mid v$ and $uv=ad$. Hence
the exponent of $K_{a,b,c,d,\ell}$ is $\lcm(u,v)=v$. We obtain
$v=\lcm(a,c)$ and $u=ad/\lcm(a,c)=\gcd(b,d)$. This gives
\eqref{H_isom}.

iv) Clear from ii). This follows also from a general result given in
\cite[Th. \ 4.2]{BauSenZve2011}.
\end{proof}

\begin{remark} {\rm One can also write $K_{a,b,c,d,\ell}$ in the form
\begin{equation*}
\left\{\left(i\frac{m}{a}, i\ell \frac{n}{c}+j\frac{n}{d}\right): 0\le i\le a-1, j_i\le j\le j_i+d-1\right\},
\end{equation*}
where $j_i= -\lfloor i\ell d/c \rfloor$. Then for the second coordinate one has $0\le i\ell n/c+jn/d\le n-1$
for every given $i$ and $j$.}
\end{remark}

\section{Number of subgroups}

According to Theorem \ref{prop_repres} the total number $s(m,n)$ of
subgroups of $\Z_m \times \Z_n$ can be obtained by counting the
elements of the set $J_{m,n}$, which is now a purely number theoretical question.

\begin{thm} \label{prop_total_number_subgroups} For every $m,n\in \N$, $s(m,n)$
is given by
\begin{equation} \label{total_number_subgroups}
s(m,n)= \sum_{i\mid m, j\mid n} \gcd(i,j)
\end{equation}
\begin{equation} \label{total_number_subgroups_var}
= \sum_{t \mid \gcd(m,n)} \phi(t) \tau\left(\frac{m}{t} \right)
\tau\left(\frac{n}{t} \right).
\end{equation}
\end{thm}

\begin{proof} We have
\begin{equation} \label{s_J}
s(m,n)= |J_{m,n}| = \sum_{\substack{a\mid m\\ b\mid a}} \sum_{\substack{c\mid
n\\ d\mid c}} \sum_{a/b=c/d=e} \phi(e).
\end{equation}

Let $m=ax$, $a=by$, $n=cz$, $c=dt$. Then, by the condition
$a/b=c/d=e$ we have $y=t=e$. Rearranging the terms of \eqref{s_J},
\begin{equation} \label{proof_line}
s(m,n)= \sum_{bxe=m} \sum_{dze=n} \phi(e) = \sum_{\substack{ix=m\\
jz=n}} \sum_{\substack{be=i\\ de=j}} \phi(e)
\end{equation}
\begin{equation*}
= \sum_{\substack{i\mid m\\
j\mid n}} \sum_{e\mid \gcd(i,j)} \phi(e) = \sum_{\substack{i\mid m\\
j\mid n}} \gcd(i,j),
\end{equation*}
finishing the proof of \eqref{total_number_subgroups}. To obtain the
formula \eqref{total_number_subgroups_var} write \eqref{proof_line}
as follows:
\begin{equation*}
s(m,n)= \sum_{\substack{ek=m\\ e\ell=n}} \phi(e)
\sum_{\substack{bx=k\\ dz=\ell}} 1 = \sum_{\substack{ek=m\\
e\ell=n}} \phi(e) \tau(k) \tau(\ell)
\end{equation*}
\begin{equation*}
= \sum_{e\mid \gcd(m,n)} \phi(e) \tau \left(\frac{m}{e}\right) \tau \left(\frac{n}{e}\right).
\end{equation*}
\end{proof}

Note that \eqref{total_number_subgroups} is a special case of an
identity deduced in \cite{Cal1987} by different arguments.

\begin{cor} \label{cor_number_subgroups} {\rm (\cite{Cal2004}, \cite[Prop.\ 2]{Pet2011},
\cite[Th.\ 3.3]{Tar2010})} The total number of subgroups of the $p$-group $\Z_{p^a}\times \Z_{p^b}$ ($1\le a\le b$) of rank two
is given by
\begin{equation} \label{prime_pow}
s(p^a,p^b)= \frac{(b-a+1)p^{a+2}-(b-a-1)p^{a+1}-(a+b+3)p+(a+b+1)}{(p-1)^2},
\end{equation}
\end{cor}

Now consider $s_{\delta}(m,n)$,
denoting the number of subgroups of order $\delta$ of $\Z_m \times
\Z_n$.

\begin{thm} \label{prop_total_number_subgroups_ord} For every $m,n,\delta \in \N$ such that $\delta \mid mn$,
\begin{equation} \label{total_number_subgroups_ord}
s_{\delta}(m,n)= \sum_{\substack{i\mid \gcd(m,\delta) \\ j\mid \gcd(n,\delta)\\
\delta\mid ij}} \phi\left(\frac{ij}{\delta}\right).
\end{equation}
\end{thm}

\begin{proof} Similar to the above proof. We have
\begin{equation*}
s_{\delta}(m,n)= \sum_{\substack{a\mid m\\ b\mid a}} \sum_{\substack{c\mid n \\
d\mid c}} \sum_{\substack{a/b=c/d=e\\ ad=bc=\delta}} \phi(e)
\end{equation*}
\begin{equation*}
= \sum_{bxe=m} \sum_{dze=n} \sum_{bde=\delta} \phi(e) = \sum_{\substack{ix=m\\
jz=n}} \sum_{\substack{be=i\\ de=j\\bde= \delta}} \phi(e),
\end{equation*}
where the only term of the inner sum is obtained for $e=ij/\delta$ provided that $\delta \mid ij$, $i\mid \delta$ and $j\mid \delta$. This gives
\eqref{total_number_subgroups_ord}.
\end{proof}

\begin{cor} \label{cor_number_subgroups_order} {\rm (\cite[Th.\ 3.3]{Tar2010})} The number of subgroups of order $p^c$ of the
$p$-group $\Z_{p^a}\times \Z_{p^b}$ ($1\le a\le b$) is given by
\begin{equation} \label{prime_pow_ord}
s_{p^c}(p^a,p^b)= \begin{cases} \frac{p^{c+1}-1}{p-1}, & c\le a\le
b,\\ \frac{p^{a+1}-1}{p-1}, & a\le c\le b,
\\ \frac{p^{a+b-c+1}-1}{p-1},
& a\le b\le c\le a+b.
\end{cases}
\end{equation}
\end{cor}

The number of subgroups of $\Z_m \times \Z_n$ with a given
isomorphism type $\Z_A \times \Z_B$ is given by the following
formula.

\begin{thm} \label{prop_number_subgroups_type} Let $m,n\in \N$ and let $A,B\in \N$ such that $A\mid B$, $AB\mid mn$. Let $A\mid \gcd(m,n)$,
Then the number $N_{A,B}(m,n)$ of subgroups of $\Z_m \times \Z_n$, which are isomorphic to $\Z_A \times \Z_B$ is given by
\begin{equation} \label{total_number_subgroups_given_type}
N_{A,B}(m,n)= \sum_{\substack{i\mid m, j\mid n \\ AB \mid ij\\
\lcm(i,j)= B}} \phi \left(\frac{ij}{AB}\right).
\end{equation}

If $A\nmid \gcd(m,n)$, then $N_{A,B}(m,n)=0$.
\end{thm}

\begin{proof} Using Proposition \ref{prop_repres}/ ii) we have
\begin{equation*}
N_{A,B}(m,n) = \sum_{\substack{a\mid m\\ b\mid a}} \sum_{\substack{c\mid n \\
d\mid c}} \sum_{\substack{a/b=c/d=e\\ \gcd(b,d)=A\\ \lcm(a,c)=B}} \phi(e).
\end{equation*}

Here the condition $\gcd(b,d)=A$ implies that $A\mid m$ and $A\mid n$. In this case
\begin{equation*}
N_{A,B}(m,n) = \sum_{\substack{bxe=m\\ dze=n}} \sum_{\substack{\gcd(b,d)=A \\ bde=AB}} \phi(e) = \sum_{\substack{ix=m\\
jz=n}} \sum_{\substack{be=i\\ de=j}} \sum_{\substack{bde= AB\\ \gcd(b,d)=A}} \phi(e),
\end{equation*}
where the only term of the inner sum is obtained for $e=ij/(AB)$
provided that $AB\mid ij$, $i\mid AB$ and $j\mid AB$.
\end{proof}

Remark: It is also possible that $N_{A,B}(m,n) = 0$ even when all
the conditions in the hypotheses are satisfied, e.g., see the Table
in Section \ref{section_table}.

From the representation of the cyclic subgroups given in Proposition
\ref{prop_repres}/ iv) we also deduce the next result.

\begin{thm} \label{prop_number_cyclic_subgroups} For every $m,n\in \N$ the number $c(m,n)$ of cyclic
subgroups of $\Z_m\times \Z_n$ is
\begin{equation} \label{total_number_cyclic_subgr_2_1}
c(m,n)=  \sum_{i \mid m, j\mid n} \phi(\gcd(i,j))
\end{equation}
\begin{equation} \label{total_number_cyclic_subgr_2_2}
= \sum_{t \mid \gcd(m,n)} (\mu*\phi)(t) \tau
\left(\frac{m}{t}\right) \tau\left(\frac{n}{t} \right).
\end{equation}
\end{thm}

\begin{proof} Similar to the above proofs, using that for the cyclic subgroups one has $\gcd(b,d)=1$.
\end{proof}

Let $c_{\delta}(m,n)$ denote the number of cyclic subgroups of order $\delta$
of $\Z_m \times \Z_n$.

\begin{thm} \label{prop_number_subgroups_cyclic_ord} For every $m,n,\delta \in \N$ such that $\delta \mid mn$,
\begin{equation*}
c_{\delta}(m,n)= \sum_{\substack{i\mid m, j\mid n\\ \lcm(i,j)= \delta}} \phi(\gcd(i,j)).
\end{equation*}
\end{thm}

\begin{proof} This is a direct consequence of \eqref{total_number_subgroups_given_type} obtained in the case $A=1$ and $B=\delta$.
\end{proof}

In the paper \cite{HHTW2012} the identities
\eqref{total_number_subgroups}, \eqref{total_number_subgroups_var},
\eqref{total_number_subgroups_ord}, \eqref{total_number_cyclic_subgr_2_1}
and \eqref{total_number_cyclic_subgr_2_2} were derived using another approach. The identity
\eqref{total_number_cyclic_subgr_2_1}, as a special case of a
formula valid for arbitrary finite abelian groups, was obtained by
the author \cite{Tot2011,Tot2012} using different arguments.
Finally, we  remark that the functions $(m,n)\mapsto s(m,n)$ and
$(m,n)\mapsto c(m,n)$ are multiplicative, viewed as arithmetic
functions of two variables. See \cite{HHTW2012,Tot2013} for details.

\section{Table of the subgroups of \texorpdfstring{$\Z_{12} \times \Z_{18}$}{Z12xZ18}}
\label{section_table}

To illustrate our results we describe the subgroups of the group
$\Z_{12} \times \Z_{18}$ ($m=12$, $n=18$). According to Proposition
\ref{prop_number_subgroups_type}, there exist subgroups isomorphic
to $\Z_A \times \Z_B$ ($A\mid B$) only if $A\mid
\gcd(12,18)=6$, that is $A\in \{1,2,3,6 \}$ and $AB\mid 12\cdot 18=216$.

\[
\vbox{\offinterlineskip \hrule \halign{ \strut
\vrule \hfill \ # \ \hfill & \vrule \hfill \ # \ \hfill \vrule \vrule &
\vrule \hfill \ # \ \hfill & \vrule \hfill \ # \ \hfill \vrule \cr
Total number subgroups & $80$ & &\cr \noalign{\hrule \, \hrule}
Number subgroups order $1$ &  1 & Number subgroups order $18$ &  12 \cr \noalign{\hrule}
Number subgroups order $2$ &  3 & Number subgroups order $24$ &  4 \cr \noalign{\hrule}
Number subgroups order $3$ &  4 & Number subgroups order $27$ &  1 \cr \noalign{\hrule}
Number subgroups order $4$ &  3 & Number subgroups order $36$ &  12 \cr \noalign{\hrule}
Number subgroups order $6$ &  12 & Number subgroups order $54$ &  3 \cr \noalign{\hrule}
Number subgroups order $8$ &  1 & Number subgroups order $72$ &  4 \cr \noalign{\hrule}
Number subgroups order $9$ &  4 & Number subgroups order $108$ &  3 \cr \noalign{\hrule}
Number subgroups order $12$ &  12 & Number subgroups order $216$ &  1 \cr \noalign{\hrule \, \hrule}
Number cyclic subgroups & $48$ & Number noncyclic subgroups & $32$ \cr \noalign{\hrule \, \hrule}
Number subgroups $\simeq \Z_1$ &  1 & Number subgroups $\simeq \Z_2 \times \Z_2$ &  1 \cr \noalign{\hrule}
Number subgroups $\simeq \Z_2$ &  3 & Number subgroups $\simeq \Z_2 \times \Z_4$ &  1 \cr \noalign{\hrule}
Number subgroups $\simeq \Z_3$ &  4 & Number subgroups $\simeq \Z_2 \times \Z_6$ &  4 \cr \noalign{\hrule}
Number subgroups $\simeq \Z_4$ &  2 & Number subgroups $\simeq \Z_2 \times \Z_{12}$ &  4 \cr \noalign{\hrule}
Number subgroups $\simeq \Z_6$ &  12 & Number subgroups $\simeq \Z_2 \times \Z_{18}$ &  3 \cr \noalign{\hrule}
Number subgroups $\simeq \Z_9$ &  3 & Number subgroups $\simeq \Z_2 \times \Z_{36}$ &  3 \cr \noalign{\hrule}
Number subgroups $\simeq \Z_{12}$ &  8 & Number subgroups $\simeq \Z_3 \times \Z_3$ &  1 \cr \noalign{\hrule}
Number subgroups $\simeq \Z_{18}$ &  9 & Number subgroups $\simeq \Z_3 \times \Z_6$ &  3 \cr \noalign{\hrule}
Number subgroups $\simeq \Z_{36}$ &  6 & Number subgroups $\simeq \Z_3 \times \Z_9$ &  1 \cr \noalign{\hrule}
& & Number subgroups $\simeq \Z_3 \times \Z_{12}$ &  2 \cr \noalign{\hrule}
& & Number subgroups $\simeq \Z_3 \times \Z_{18}$ &  3 \cr \noalign{\hrule}
& & Number subgroups $\simeq \Z_3 \times \Z_{36}$ &  2 \cr \noalign{\hrule}
& & Number subgroups $\simeq \Z_6 \times \Z_6$ &  1 \cr \noalign{\hrule}
& & Number subgroups $\simeq \Z_6 \times \Z_{12}$ &  1 \cr \noalign{\hrule}
& & Number subgroups $\simeq \Z_6 \times \Z_{18}$ &  1 \cr \noalign{\hrule}
& & Number subgroups $\simeq \Z_6 \times \Z_{36}$ &  1   \cr} \hrule}
\]
\medskip
\centerline{Table of the subgroups of $\Z_{12} \times \Z_{18}$}


\section{Acknowledgements} The author gratefully acknowledges support
from the Austrian Science Fund (FWF) under the project Nr.
M1376-N18. The author thanks D.~D.~Anderson and J.~Petrillo for
sending him copies of their papers quoted in the References. The author is thankful to the referee
for careful reading of the manuscript and for many helpful suggestions on the presentation of this paper.


\noindent L. T\'oth \\
Institute of Mathematics, Universit\"at f\"ur Bodenkultur \\
Gregor Mendel-Stra{\ss}e 33, A-1180 Vienna, Austria \\ and \\
Department of Mathematics, University of P\'ecs \\ Ifj\'us\'ag u. 6,
H-7624 P\'ecs, Hungary \\ E-mail: {\tt ltoth@gamma.ttk.pte.hu}


\begin{thebibliography}{99}

\bibitem{AndCam2009} D.~D.~Anderson, V.~Camillo, Subgroups of direct products of
groups, ideals and subrings of direct products of rings, and
Goursat's lemma. Rings, modules and representations, {\it Contemp.
Math.}, vol. 480, Amer. Math. Soc., Providence, RI, 2009, pp. 1--12.

\bibitem{BauSenZve2011} K.~Bauer, D.~Sen, P.~Zvengrowski, A generalized Goursat lemma, Preprint, 2011,
arXiv: 11009.0024 [math.GR].

\bibitem{Cal1987} W.~C.~Calhoun, Counting the subgroups of some finite groups, {\it Amer. Math. Monthly}
{\bf 94} (1987), 54--59.

\bibitem{Cal2004} G.~C\u{a}lug\u{a}reanu, The total number of subgroups of a finite
abelian group, {\it Sci. Math. Jpn.} {\bf 60} (2004), 157--167.

\bibitem{CraWal1975} R.~R.~Crawford, K.~D.~Wallace, On the number of subgroups of
index two --- An application of Goursat's theorem for groups, {\it
Math. Mag.} {\bf 48} (1975), 172--174.

\bibitem{Gou1889} \'E.~Goursat, Sur les substitutions orthogonales et les divisions
r\'eguli\`{e}res de l'espace, {\it Ann. Sci. \`{E}cole Norm. Sup.}
(3) {\bf 6} (1889), 9--102.

\bibitem{HHTW2012} M.~Hampejs, N.~Holighaus, L.~T\'oth, C.~Wiesmeyr, Representing and counting the subgroups
of the group $\Z_m \times \Z_n$, {\it Journal of Numbers}, Volume 2014, Article ID 491428.

\bibitem{HamTot2013} M.~Hampejs, L.~T\'oth, On the subgroups of finite abelian groups of rank three,
{\it Annales Univ. Sci. Budapest., Sect Comp.} {\bf 39} (2013),
111--124.

\bibitem{Lam1958} J.~Lambek, Goursat's theorem and the Zassenhaus lemma, {\it Canad.
J. Math.} {\bf 10} (1958), 45--56.

\bibitem{Mac2012} A.~Mach\`{i}, {\it Groups. An Introduction to Ideas and Methods
of the Theory of Groups}, Springer, 2012.

\bibitem{NowTot2013} W.~G.~Nowak, L.~T\'oth, On the average number of subgroups of the
group $\Z_m \times \Z_n$, {\it Int. J. Number Theory} {\bf 10} (2014), 363--374,
arXiv:1307.1414 [math.NT].

\bibitem{Pet2009} J.~Petrillo, Goursat's other theorem, {\it College Math. J.} {\bf
40} (2009), 119--124.

\bibitem{Pet2011} J.~Petrillo, Counting subgroups in a direct product of finite cyclic groups,
{\it College Math. J.} {\bf 42} (2011), 215--222.

\bibitem{Rot1995} J.~J.~Rotman, {\it An Introduction to the Theory of Groups}, Fourth Ed., Springer, 1995.

\bibitem{Tar2010} M.~T\u{a}rn\u{a}uceanu, An arithmetic method of counting the
subgroups of a finite abelian group, {\it Bull. Math. Soc. Sci.
Math. Roumanie (N.S.)} {\bf 53(101)} (2010), 373--386.

\bibitem{Tot2011} L.~T\'oth, Menon's identity and arithmetical sums
representing functions of several variables, {\it Rend. Sem. Mat.
Univ. Politec. Torino} {\bf 69} (2011), 97--110.

\bibitem{Tot2012} L.~T\'oth, On the number of cyclic subgroups of a finite Abelian
group, {\it Bull. Math. Soc. Sci. Math. Roumanie (N.S.)} {\bf
55(103)} (2012), 423--428.

\bibitem{Tot2013} L.~T\'oth, Multiplicative arithmetic functions of several variables:
a survey, in {\it Mathematics Without Boundaries, Surveys in Pure Mathematics},
Th.~M.~Rassias, P.~Pardalos (Eds.), Springer, New York, 2014, 483--514, arXiv:1310.7053 [math.NT].
\end{thebibliography}
\end{document}